\documentclass[english]{ourlematema}
\pdfoutput=1
\usepackage{amsmath}
\usepackage{amsfonts}
\usepackage{amssymb}
\usepackage{amsthm}
\usepackage{hyperref}

\usepackage{rotating}
\usepackage{xcolor}

\newtheorem{theorem}{Theorem}
\numberwithin{theorem}{section}
\newtheorem{propo}[theorem]{Proposition}

\newtheorem{obs}[theorem]{Observation}
\newtheorem{mylemma}[theorem]{Lemma}

\theoremstyle{definition}
\newtheorem{definition}[theorem]{Definition}

\theoremstyle{remark}
\newtheorem{remark}[theorem]{Remark}
\numberwithin{equation}{section}

\DeclareMathOperator{\ann}{Ann}

\DeclareMathOperator{\Inv}{Inv}
\DeclareMathOperator{\Hu}{Hu}

\def\CC{\mathbb{C}}

\def\PP{\mathbb{P}}
\def\NN{\mathbb{N}}

\def\GG{\mathbb{G}}
\def\fC{\mathcal{C}}
\def\fH{\mathcal{H}}

\title{The Hessian Discriminant}

\author{Rodica Dinu}
\address{%
	University of Bucharest\\
	\email{rdinu@fmi.unibuc.ro}
}

\author{Tim Seynnaeve}
\address{%
	Max Planck Institute for Mathematics in the Sciences\\
	\email{tim.seynnaeve@mis.mpg.de}
}

\date{2019/09/14}

\begin{document}

\maketitle

\begin{abstract}
	We express the Hessian discriminant of a cubic surface in terms of fundamental invariants. This answers Question 15 from the \emph{27 questions on the cubic surface}. We also explain how to compute the fundamental invariants for smooth cubics of rank 6. 
\end{abstract}
\section*{Introduction}
The \emph{Hessian discriminant} \cite{Seigal} is a locus of cubic surfaces whose complex rank jumps from $5$ to $6$. It is defined as the vanishing set of a homogeneous degree $120$ polynomial $HD$ in $20$ variables,
which is a specialization of the Hurwitz form of the variety of rank $2$ matrices in $\PP(S^2(\CC^4))$.

By construction, the Hessian discriminant is a hypersurface in $\PP^{19}$ which is invariant under $PGL(3)$. In other words, $HD$ is an invariant of cubic surfaces. 
The generators of the invariant ring of cubic surfaces are known \cite{Salmon}, so it is natural to ask how to express $HD$ in terms of these fundamental invariants. This was Question 15 in the \emph{27 questions on the cubic surface} \cite{27questions}. The main result of this article provides an answer to this question:
 \begin{theorem} \label{thm:main}
	$HD=I_{40}^3$, where $I_{40}$ is the degree 40 Salmon invariant. 
\end{theorem}
The organization of the article is as follows:

In Section \ref{sec:hd}, we review the definition of the Hurwitz form and the Hessian discriminant. We also explain how to use software to verify whether a given cubic lies on the Hessian discriminant, and explain a connection with apolar schemes.

In Section \ref{sec:normalForms}, we use the classical theory of normal forms for cubic surfaces to decide for every cubic surface outside of a certain codimension 2 locus whether or not it lies on the Hessian discriminant.

In Section \ref{sec:fundinv}, we recall the invariant theory of cubic surfaces, and give a computational proof that the vanishing locus of the invariant $I_{40}$ consists of the smooth rank 6 cubic surfaces.

Finally, in Section \ref{sec:mainproof} we put together the results of the preceding two sections to prove Theorem \ref{thm:main}.
\subsection*{Acknowledgements}
The authors would like to thank Bernd Sturmfels for creating the list of 27 questions on the cubic surface.
They would also like to thank Anna Seigal and Kristian Ranestad for helpful comments and discussions.
\section{The Hessian discriminant} \label{sec:hd}
\subsection{The Hurwitz form}
Let $X$ be an irreducible variety in projective space $\PP^n$ of codimension $d\geq 1$ and degree $p \geq 2$. Let $\GG(d,\PP^n)$ denote the Grassmannian of dimension $d$ subspaces of $\PP^n$. Following  \cite{sturmfelsHurwitz}, we define $\mathcal{H}_X \subset \GG(d,\PP^n)$ to be the set of all subspaces $L$ for which $L \cap X$ does not consist of $p$ reduced points. 

If $L$ is the row space of a matrix $B=(b_{i,j})_{0\leq i\leq d, 0\leq j \leq n}$, then the entries $b_{i,j}$ are the {\it Stiefel coordinates} of $L$, and the maximal minors of $B$ are the {\it Pl\"ucker coordinates}.

One can obtain the {\it sectional genus} of $X$ by intersecting the variety with a general subspace of codimension $d-1$ and then taking the arithmetic genus of the obtained curve.
\begin{theorem}\cite[Theorem 1.1]{sturmfelsHurwitz}\label{bernd}
	$\mathcal{H}_X$ is an irreducible hypersurface in $\GG(d,\PP^n)$, defined by an irreducible element $\Hu_X$ in the coordinate ring of $\GG(d,\PP^n)$. If $X$ is regular in codimension 1, then the degree of $\Hu_X$ in Pl\"ucker coordinates equals $2p + 2g - 2$, where $g$ is the sectional genus of $X$.
\end{theorem}
The polynomial $\Hu_X$ defined above is called the \emph{Hurwitz form} of $X$. 
Interesting examples of Hurwitz forms in computational algebraic geometry can be consulted in \cite{sturmfelsHurwitz}.

To define the Hessian discriminant, we will need to consider the Hurwitz form of the variety $X_2$ of rank $2$ symmetric $4 \times 4$ matrices. If we write $\PP^{9}$ for the space of all symmetric $4 \times 4$ matrices, then $X_2 \subset \PP^9$ is an irreducible subvariety defined by the vanishing of the $3\times 3$ minors. It has dimension 6, degree 10, and sectional genus 6. By Theorem \ref{bernd}, the Hurwitz form $\Hu_{X_2}$ is an irreducible hypersurface of degree 30 in the Pl\"ucker coordinates of $\GG(3, \PP^9)$.  
In \cite{sturmfelsHurwitz}, there is an algorithm to compute the polynomial $\Hu_X$, but it does not finish in a reasonable amount of time in this case.

\subsection{The Hessian discriminant}
For the rest of the paper, we fix a 4-dimensional $\CC$-vector space $V$. Let $\fC=V(f)$ be a cubic surface in $\PP^3=\PP(V)$, defined by a quaternary cubic
\[
f = \sum_{0 \leq i \leq j \leq k \leq 3}{c_{ijk}x_ix_jx_k} \in \CC[x_0,x_1,x_2,x_3]_3 = S^3(V^*).
\]
If $\fC$ is not a cone over a plane cubic, we can associate to $f$ a 3-plane $H(f)$ in the space $\PP^9 = \PP(S^2(V^*))$ of symmetric $4 \times 4$ matrices. The points of $H(f)$ are called \emph{polar quadrics} of $f$. There are several equivalent ways to define $H(f)$. We leave it to the reader to check that they are indeed equivalent.
\begin{itemize}
	\item The Hessian of $f$ is the $4 \times 4$ matrix of linear forms whose $(i,j)$'th entry is $\frac{\partial^2 f}{\partial x_i \partial x_j}$. It defines a linear map $\tilde{f}: \PP^3 \to \PP^9$, sending a point $p=[x_0:x_1:x_2:x_3]$ to the Hessian matrix evaluated in that point. We define $H(f)$ to be the image of $\tilde{f}$.
	\item We can also define $H(f)$ as the linear span of the four partial derivatives $\frac{\partial f}{\partial x_0}$,$\frac{\partial f}{\partial x_1}$, $\frac{\partial f}{\partial x_2}$, $\frac{\partial f}{\partial x_3}$, seen as points in $\PP(S^2(V^*))$. Note that these 4 points are well-defined and not coplanar, unless after change of coordinates $f$ is a polynomial in $3$ variables. This explains our assumption that $\fC$ is not a cone over a plane cubic.
	\item We can view $f$ as a symmetric three-way tensor $T=(T_{ijk})_{i,j,k}$. (I.e.\ $c_{ijk}=\lambda T_{ijk}$, where $\lambda$ is the number of distinct permutations of $i,j,k$. Then $f=\sum_{i,j,k}T_{ijk}x_ix_jx_k$.) For $m \in \{0,1,2,3\}$, the $m$'th slice of $T$ is defined to be the symmetric matrix obtained by fixing the first index to be $m$. Then $H(f)$ is the linear span of the four slices of $T$. From this description we see immediately the Stiefel coordinates of $H(f)$: they are the entries of a $4 \times 10$ matrix with colums indexed by pairs $(j,k)$ with $j < k$, whose $i,(j,k)$'th entry is $T_{ijk}$.
\end{itemize}
 Now we can take the Hurwitz polynomial $\Hu_{X_2}$ from the previous subsection, and evaluate it in the Pl\"ucker coordinates of $H(f)$, where $f$ is a general cubic surface. The result is a degree 120 polynomial in the 20 variables $c_{ijk}$, called the \emph{Hessian discriminant}  $HD$. 
 By construction, the Hessian discriminant vanishes at $f \in \PP^{19}$ if and only if $H(f)$ does not intersect the variety of rank 2 matrices in 10 reduced points.
 Clearly, $V(HD)$ is invariant under linear changes of coordinates. 
 It follows that $HD$ is invariant under the natural action of $SL(4)$ on $\CC[c_{000},\ldots,c_{333}]$. 
  
  We close this subsection with some interesting observations about the Hessian discriminant. 
  The determinant of the Hessian of $f$ defines a quartic surface in $\PP^3$, called the \emph{Hessian surface} of $f$.
 It can be identified with the intersection of $H(f)$ and the variety $X_3$ of singular $4 \times 4$ matrices. Since the singular locus of $X_3$ is equal to $X_2$, rank 2 matrices give rise to singular points of the Hessian surface. As shown in \cite{Dardanelli}, when $V(f)$ is smooth, these are the only singularities, but if $V(f)$ is singular, its Hessian surface has additional singularities. Thus we have shown the following:
  \begin{obs}
  	A smooth cubic surface does not lie on the Hessian discriminant if and only if its Hessian surface has precisely 10 singular points.
  \end{obs}
  
 Recall that the \emph{Waring rank} of a degree $d$ polynomial $f$ is the smallest $r$ such that $f$ can be expressed as a sum of $r$ $d$'th powers of linear forms. The following observation connects the Waring rank of cubic forms with the Hessian discriminant:
 \begin{obs} \label{obs:SeigalRank6} ({See \cite{Seigal}.})
 	If $f$ has Waring rank at least 6 and defines a smooth cubic surface, then $f$ lies on the Hessian discriminant.
 \end{obs}
It will be easy to verify this, once we have recalled the normal forms of smooth cubic surfaces. 

\subsection{Computational methods}
While computing an expression for the Hessian discriminant is a computationally difficult task, it is easy to verify for a given cubic whether or not the Hessian discriminant vanishes at that cubic: one simply needs to compute the ideal defining the intersection of $H(f) \cap X_2$, and check whether or not it is zero-dimensional and radical. Some code in \verb|Macaulay2| \cite{M2} for computing this can be found below:
\begin{verbatim}
R=QQ[x_0..x_3,z_0..z_9]
X={x_0,x_1,x_2,x_3};
A=genericSymmetricMatrix(R,z_0,4)
I2=minors(3,A)
hessRank2 = f ->(
    hess = diff(transpose matrix{X},diff(matrix{X},f));
    I=eliminate(X,ideal(flatten entries (A-hess)));
    return (I+I2);
)
isOnHessianDiscriminant = f ->(
    J=hessRank2(f);
    return not ((codim J==9) and (J==radical J));
)
--Examples:
f=x_0*x_1*x_2+x_0*x_1*x_3+x_0*x_2*x_3+x_1*x_2*x_3
isOnHessianDiscriminant(f)
--false
f=x_0^3+x_1^3+x_2^3+x_3^2*(3*x_0+3*x_1+3*x_2+x_3)
isOnHessianDiscriminant(f)
--true
\end{verbatim}
\begin{remark} \label{rmk:algo}
 The above algorithm can also be used to simultaneously compute $H(f) \cap X$ for all $f$ in a family of cubic surfaces. For more details, see Remark \ref{rmk:algoSylv}, as well as the supplementary code available at \cite{webpageTim}.
\end{remark}
\subsection{Apolarity}
There is a beautiful connection between the 3-plane $H(f)$ associated to $f$ and apolar schemes of $f$. Although not logically necessary for the proof of our main theorem, it can provide some insight in the nature of the singularities of the Hessian surface of a smooth cubic.

We will identify the symmetric algebra $S(V)$ of $V$ with the polynomial ring $\CC[y_0,y_1,y_2,y_3]$. For every $d,m \in \NN$, there is a natural pairing
	\[
\circ: \CC[y_0,y_1,y_2,y_3]_d \times \CC[x_0,x_1,x_2,x_3]_m \to \CC[x_0,x_1,x_2,x_3]_{m-d}
\] 
 defined by $g \circ f = g(\frac{\partial}{\partial x_0},\frac{\partial}{\partial x_1},\frac{\partial}{\partial x_2},\frac{\partial}{\partial x_3})f(x_0,x_1,x_2,x_3)$. 
 
 Note that $H(f)$ can be identified with the image of the map $V \to S^2(V^*): g \mapsto g \circ f$.
\begin{definition}
	For $f$ in $\CC[x_0,x_1,x_2,x_3]$, we define the \emph{annihilator} of $f$ to be the ideal 
	\[
	\ann(f) = \{g| g \circ f = 0\} \subseteq \CC[y_0,y_1,y_2,y_3].
	\]
	If $I \subseteq \ann(f)$ is a saturated ideal, we say that $I$ is an \emph{apolar ideal} to $f$, and $V(f)$ is an \emph{apolar scheme} to $f$. In other words, $Y \subseteq \PP(V^*)$ is an apolar scheme to $f$ if every polynomial that vanishes on $Y$ also annihilates $f$.
\end{definition}
\begin{obs} \label{obs:AnniHess}
	Denote the coordinates on $\PP^9=S^2(V^*)$ by $z_{ij}$.
	Then defining equations $\sum_{i\leq j}a_{ij}z_{ij}=0$ of $H(f)$ are in one-to-one correspondence with degree 2 elements $\sum_{i\leq j}a_{ij}y_iy_j$ of $\ann(f)$:
	$\sum_{i\leq j}a_{ij}y_iy_j$ is in $\ann(f)$ if and only if $\sum_{i\leq j}a_{ij}\frac{\partial^2f}{\partial x_i \partial x_j}=0$, if and only if $\sum_{i\leq j}a_{ij}z_{ij}$ vanishes on $H(f)$. 
	As a corollary of this, if $Y$ is an apolar scheme to $f$, then $H(f)$ is contained in the linear span of $v_2(Y)$,  the image of $Y$ under the second Veronese embedding $v_2: \PP(V^*) \to \PP(S^2(V^*))$.
Indeed: every linear equation $\sum_{i\leq j}a_{ij}z_{ij}$ on $v_2(Y)$ comes from a quadratic equation $\sum_{i\leq j}a_{ij}y_iy_j$ on $Y$, which by the above also vanishes on $H(f)$.
\end{obs}
\section{Normal forms for cubics} \label{sec:normalForms}
It is possible to classify the cubic surfaces up to linear transformation, and use this to provide a list of normal forms so that every quaternary cubic can be brought in one of the normal forms by a linear change of coordinates. This was first done by Schl\"afli \cite{Schlafli}, we refer the interested reader to \cite{Schmitt} for an overview, and to \cite{segre} and \cite{BruceWall} for detailed accounts on the smooth, respectively the singular case.

For our purposes, it suffices to know the following result:
\begin{propo} \label{prop:Normalforms}
	Every cubic, outside of a certain codimension $>1$ set in $\PP^{19}$, can after a linear change of coordinates be written in one of the following 3 normal forms:
	\begin{enumerate}
		\item Sylvester's Pentahedral form:
		\begin{equation}\label{sylv}
		c_0x_0^3+c_1x_1^3+c_2x_2^3+c_3x_3^3+c_4(-x_0-x_1-x_2-x_3)^3=0,
		\end{equation}
		where $c_i \in \CC^*$, and $\sum_{i}{\pm \frac{1}{\sqrt{c_i}}} \neq 0$.
		\item Generic rank 6 cubics:
		\begin{equation}\label{rank6}
		(x_0^3 + x_1^3 + x_2^3) - x_3^2(3\lambda_0x_0 + 3\lambda_1x_1 + 3\lambda_2x_2 - x_3)=0,
		\end{equation}
		where $\lambda_i \in \CC^*$, and $2(\lambda_0^{\frac{3}{2}}+\lambda_1^{\frac{3}{2}}+\lambda_2^{\frac{3}{2}})\neq 1$.
		\item Generic singular cubics:
		\begin{equation}\label{sing}
		 x_3(x_1^2-x_0x_2)+x_1(x_0-(1+\rho_0)x_1+\rho_0x_2)(x_0-(\rho_1+\rho_2)x_1+\rho_1\rho_2x_2)=0,
		\end{equation}
		where $\rho_i \in \CC \setminus \{0,1\}$ pairwise different.
	\end{enumerate}
	The set of cubics that can be brought in the form (\ref{sylv}) is dense in $\PP^{19}$; the sets of cubics that can be brought in the form (\ref{rank6}) respectively (\ref{sing}) are both of codimension one in $\PP^{19}$.
\end{propo}
\begin{proof}[Idea of the proof]
	We call $f \in \CC[x_0,x_1,x_2,x_3]_3$
	\begin{itemize}
		\item \emph{non-degenerate Sylvester} if $f$ can be written as a sum of 5 cubes in a unique way.
		\item \emph{cyclic} if $f$ can be written as a sum of 5 cubes in more than one way.
		\item \emph{non-Sylvester} if $f$ cannot be written as a sum of 5 cubes.
	\end{itemize}
	It can be shown (see \cite{segre}) that these are families of respective dimensions $19$, $16$ and $18$ in $\PP^{19}$, and that a general non-Sylvester cubic can be brought in the form (\ref{rank6}).
	
	A general singular cubic can be written in the form (\ref{sing}): this is proven in \cite[Lemma 2]{BruceWall}, see also \cite[Theorem 2]{Schmitt}.
\end{proof}
\subsection{Sylvester's pentahedral form}
\begin{propo} \label{prop: SylvHu}
	No cubic of the form (\ref{sylv}) lies on the Hessian discriminant, as long as all $c_i$ are nonzero.
\end{propo}
\begin{proof}
	Write $x_4=-x_0-x_1-x_2-x_3$.
	The set of rank 2 quadratic forms in $H(f)$ is given by
	\[
	\{c_ix_i^2-c_jx_j^2|i \leq j\}.
	\]
	This can easily be verified by hand, or computationally by using the algorithm described below.
\end{proof}
\begin{remark} \label{rmk:algoSylv}
	We can use our \verb|Macaulay2| code to simultaneously analyze $H(f) \cap X_2$ for all cubics $f$ of the form (\ref{sylv}), including the ones where one of the $c_i$ is zero (if 2 or more of them are zero then $H(f)$ is not defined). The code (available at \cite{webpageTim}) computes a primary decomposition of the ideal defining $H(f) \cap X_2$ (where the $c_i$ are variables). The primary decomposition of our ideal has 40 components. 30 of these contain one of the paramaters $c_0,\dots,c_4$ (each parameter in 6 components); the other 10 do not contain any linear combination of the parameters.
	This means that if exactly one of the parameters is zero, the interction $H(f) \cap X_2$ consists of 6 lines, whereas if all the $c_i$ are nonzero, it consists of 10 points.
\end{remark}

\begin{remark} \label{rmk: generalCubic10pts}
	Propostition \ref{prop: SylvHu} can also be shown using apolarity: for a general cubic surface
	\[
	f := L_1^3+L_2^3+L_3^3+L_4^3+L_5^3 =0
	\]
	(with the $L_i$ in general position)
	there is an apolar scheme $Y=\{L_1,\ldots,L_5\}$. This can easily be verified directly, but also follows from the so-called \emph{apolarity lemma}, which states that a homogeneous degree $d$ polynomial $f$ can be written as a linear combination of powers $L_1^d , \ldots, L_s^d$ of linear forms if and only if $\{L_1,\ldots,L_s\}$ is an apolar scheme to $f$.
	It now follows from Observation \ref{obs:AnniHess} that $H(f)$ is contained in the linear span $\langle v_2(Y) \rangle$ of the second Veronese embedding of $Y$.
	
	Clearly, $\langle v_2(Y) \rangle \cap X_2$ contains the 10 lines through the $\langle L_i^2, L_j^2 \rangle$, and (using the fact that $L_i$ are in general position) it is easy to verify that this is in fact an equality. 
	Now, $H(f) \cap X_2$ is the intersection of $\langle v_2(Y) \rangle \cap X_2$ with a hyperplane $H$. Since $H(f)$ does not contain any of the $L_i^2$ (indeed: this would imply that there is a $g$ such that $g \circ L_i=0$ for 4 out 5 of the $L_i$, contradicting the general position), $H$ intersects every line $\langle L_i^2,L_j^2 \rangle$ in one point, and these points are distinct. These are the 10 points of $H(f) \cap X$.
\end{remark}
\subsection{Rank six cubics}
\begin{propo} \label{prop:rank67points}
	For a cubic $f$ of the form (\ref{rank6}), the scheme-theoretic intersection $H(f) \cap X_2$ consists of 4 simple and 3 double points. In particular, $f$ lies on the Hessian discriminant.
\end{propo}
\begin{proof}
	 The scheme $H(f) \cap X_2$ is supported at the 7 points $x_0^2-\lambda_0x_3^2$, $x_1^2-\lambda_1x_3^2$, $x_2^2-\lambda_2x_3^2$, $\lambda_1x_0^2-\lambda_0x_1^2$, $\lambda_2x_0^2-\lambda_0x_2^2$, $\lambda_2x_1^2-\lambda_1x_2^2$, $x_3(\lambda_0x_0+\lambda_1x_1+\lambda_2x_2-x_3)$, where the first three are double points. This can be verified using our code \cite{webpageTim}.
\end{proof}
\begin{remark}
	There is an intuitive explanation why $H(f) \cap X_2$ contains 3 double points. As we will see in Section \ref{subsec:rank6inv}, a cubic of the form (\ref{rank6}) can be obtained as a limit of cubics of the form $\sum_{i=1}^{5}{L_i^3}$, where in the limit the points $L_4, L_5 \in \PP^3$ crash together. In Remark \ref{rmk: generalCubic10pts} we saw that for cubics in pentahedral form, the 10 points in $H(f) \cap X_2$ are in bijection with the 10 lines between the 5 points $L_i^2 \in \PP^9$. Now if 2 of our points crash together, these 10 lines become 4 single lines and 3 double lines. We will now make this more precise.
	 
	For a general rank 6 cubic surface
	\[
	f := L_1^3+L_2^3+L_3^3+L_4^2M =0
	\]
	(with $L_1,L_2,L_3L_4,M$ in general position)
	let $Z$ be the nonreduced scheme of length 2 supported at $L_4$ in direction $M$, i.e.\ $I(Z) = I(L_4)^2 + I( \langle L_4,L_5 \rangle )$.
	Then $Y=L_1 \cup L_2 \cup L_3 \cup Z$ is a length 5 apolar scheme of $f$. Hence $H(f) \subset \langle v_2(Y) \rangle$.

	Note that $v_2(Y)=\langle L_1^2,L_2^2,L_3^2,L_4^2,L_4M\rangle$. From this we can see that $\langle v_2(Y) \rangle \cap X_2$ contains the 4 lines $\langle L_1^2,L_2^2 \rangle$, $\langle L_1^2,L_3^2 \rangle$, $\langle L_2^2,L_3^2 \rangle$ and $\langle L_4^2,L_4M \rangle$, as well as three double lines defined by $I(\langle L_i^2,L_4^2 \rangle^2) + I(\langle L_i,L_4,M \rangle^2)$. As before, from the fact that $L_i$ and $M$ are general we can see that $\langle v_2(Y) \rangle \cap X_2$ consists precisely of these lines.
	
	Now, $H(f) \cap X$ is the intersection of $\langle v_2(Y) \rangle \cap X_2$ with a hyperplane $H$. Since $H(f)$ does not contain any of the $L_i^2$, $H$ intersects every of our 7 lines in 1 (possibly fat) point. These are the 7 points of $H(f) \cap X_2$.
\end{remark}
\begin{remark}
	It is proven in \cite{segre} that \emph{every} smooth cubic surface of Waring rank at least six can either be brought in the form (\ref{rank6}), or in the following form:
	\[
    \mu_0x_0^3+x_1^3+x_2^3-3x_0(\mu_1x_0x_1+x_0x_2+\mu_2x_3^2)=0.
	\]
	For cubics of this form, we can use our code to show that $H(f) \cap X_2$ consists of 3 triple points and one single point. In particular, they lie on the Hessian discriminant, and we recover Observation \ref{obs:SeigalRank6}.
\end{remark}
\subsection{Generic singular cubics}
\begin{propo} \label{prop:singHur}
	A general singular cubic does not lie on the Hessian discriminant.
\end{propo}
\begin{proof}
	It suffices to find one singular cubic that does not lie on the Hessian discriminant. One way of doing this is by generating a random one and using our code. Here we will instead exhibit a very specific example: the Cayley cubic, given by
	\[
	f=x_0x_1x_2+x_0x_1x_3+x_0x_2x_3+x_1x_2x_3.
	\]
	Then $H(f)$ is the linear span of the quadratic forms 
	\[
	x_1x_2+x_1x_3+x_2x_3, x_0x_2+x_0x_3+x_2x_3, x_0x_1+x_0x_3+x_1x_3, x_0x_1+x_0x_2+x_1x_2, 
	\]
	and $H(f) \cap X_2$ consists of the following 10 distinct points: 
	\begin{eqnarray*}
x_0(x_1+x_2+x_3), x_1(x_0+x_2+x_3), x_2(x_0+x_1+x_3), x_3(x_0+x_1+x_2),\\ (x_0-x_1)(x_2+x_3), (x_0-x_2)(x_1+x_3), (x_0-x_3)(x_1+x_2),\\
 (x_1-x_2)(x_0+x_3), (x_1-x_3)(x_0+x_2), (x_2-x_3)(x_0+x_1). 
	\end{eqnarray*}
	This shows that the $f$ does not lie on the Hessian discriminant.
	\end{proof}
\section{Fundamental invariants} \label{sec:fundinv}
The natural action of $SL(4)$ on $V$ induces an action on the space $S^3(V^*)$ of quaternary cubics, which in turn induces an action on the polynomial ring $R=\CC[c_{000},\ldots, c_{333}]$ in the 20 coefficients of a quaternary cubic.
Then the invariant ring $R^{SL(4)}$ is the ring of all polynomials in the coefficients of a cubic surface that are invariant under a (determinant 1) linear change of coordinates.
 It was shown by Salmon \cite{Salmon} that $R^{SL(4)}$ is generated by irreducible polynomials $I_8,I_{16},I_{24},I_{32},I_{40},I_{100}$, where $I_d$ has degree $d$.  The first 5 are algebraically independent and $I_{100}^2$ can be written as a polynomial in $I_8,I_{16},I_{24},I_{32},I_{40}$. The expressions for $I_d$ in terms of $c_{ijk}$ are hard to obtain and too long to write down here. However, it is easy to write them down for cubics in Sylvester normal form.
 
 For a cubic of the form (\ref{sylv}), we write
 \begin{eqnarray*}
	&\sigma_1=&c_0+c_1+c_2+c_3+c_4 \\
	&\sigma_2 =&c_0c_1+c_0c_2+c_0c_3+c_0c_4+c_1c_2+c_1c_3+c_1c_4+c_2c_3+c_2c_4+c_3c_4\\
	&\sigma_3=&c_0c_1c_2+c_0c_1c_3+c_0c_1c_4+c_0c_2c_3+c_0c_2c_4\\
	& &+c_0c_3c_4+c_1c_2c_3+c_1c_2c_4+c_1c_3c_4+c_2c_3c_4\\
	&\sigma_4=&c_0c_1c_2c_3+c_0c_1c_2c_4+c_0c_1c_3c_4+c_0c_2c_3c_4+c_1c_2c_3c_4\\
	&\sigma_5=&c_0c_1c_2c_3c_4.
\end{eqnarray*}
 Then we can write the fundamental invariants as follows:
 \begin{eqnarray*}
	&I_8=&\sigma_4^2-4\sigma_3\sigma_5 \\
	&I_{16}=&\sigma_5^3\sigma_1\\
	&I_{24}=&\sigma_5^4\sigma_4\\
	&I_{32}=&\sigma_5^6\sigma_2\\
	&I_{40}=&\sigma_5^8.
\end{eqnarray*}
\begin{remark}
	The tuple $[I_8,I_{16},I_{24},I_{32},I_{40}]$ gives a point in weighted projective space $\PP(1,2,3,4,5)$. 
	We will denote this point by $\Inv(\fC)$.
\end{remark}
\subsection{Computing invariants for cubics of higher rank} \label{subsec:rank6inv}
A general cubic of the form (\ref{rank6}) has Waring rank 6 \cite{Seigal}, i.e.\ cannot be written as a sum of 5 cubes.
However, since a generic quaternary cubic can be brought in Sylvester normal form, \emph{any} quaternary cubic $\fC$ can be arbitrarily closely approximated by cubics in Sylvester normal form.

We do this for cubics of the form (\ref{rank6}). Fix a cubic $\fC$ with equation
\[
f(x_0,x_1,x_2,x_3) := (x_0^3 + x_1^3 + x_2^3) - x_3^2(3\lambda_0x_0 + 3\lambda_1x_1 + 3\lambda_2x_2 - x_3)=0.
\]
 For every $\varepsilon \in \CC^*$, we define a cubic $\fC_\epsilon$ with equation
\begin{eqnarray*}
	f_{\varepsilon}(x_0,x_1,x_2,x_3):=\frac{1}{\lambda_0^3\varepsilon^3}(\varepsilon\lambda_0x_0)^3+\frac{1}{\lambda_1^3\varepsilon^3}(\varepsilon\lambda_1x_1)^3+\frac{1}{\lambda_2^3\varepsilon^3}(\varepsilon\lambda_2x_2)^3 +\\ (1+\frac{1}{\varepsilon})x_3^3+\frac{1}{\varepsilon}(-\varepsilon\lambda_0x_0-\varepsilon\lambda_1x_1-\varepsilon\lambda_2x_2-x_3)^3=0.
\end{eqnarray*}
Note that $\lim_{\varepsilon \to 0} f_\varepsilon = f$.

 For every fixed $\varepsilon$, we can compute $\Inv(\fC_\varepsilon) \in \PP(1,2,3,4,5)$ (see our code at \cite{webpageTim}). Taking the limit $\varepsilon \to 0$ gives us 
 \begin{eqnarray*}
 	\Inv(\fC) =& [1-4(\lambda_0^3+\lambda_1^3+\lambda_2^3): \lambda_0^3\lambda_1^3+\lambda_0^3\lambda_2^3+\lambda_1^3\lambda_2^3: 2\lambda_0^3\lambda_1^3\lambda_2^3:\\ & \lambda_0^3\lambda_1^3\lambda_2^3(\lambda_0^3+\lambda_1^3+\lambda_2^3): 0]
 \end{eqnarray*}
In particular, we can deduce the following result (see also \cite[Chapter 9.4.5]{Dolgachev}):
\begin{propo} \label{prop:I40is0forrank6}
	For a general smooth cubic of rank 6, it holds that $I_{40}=0$.
\end{propo}
\section{Proof of the main theorem} \label{sec:mainproof}
 \begin{theorem}
	Let $HD$ be the degree 120 polynomial in $c_{000},\ldots,c_{333}$ obtained by evaluating the Hurwitz form $\Hu_X$ of the variety of rank 2 matrices in the Pl\"ucker coordinates of $H(f)$, where $f$ defines a general cubic surface. Then $HD=I_{40}^3$, where $I_{40}$ is the degree 40 Salmon invariant. 
\end{theorem}
\begin{proof}
	We will show that $V(HD) = V(I_{40})$. Then the result follows from the fact that $\deg(HD) = 3 \deg(I_{40})$.
	
	The set of cubics that can be brought in the form (\ref{rank6}) is of codimension one by Proposition \ref{prop:Normalforms}, lies on the Hurwitz form by Proposition \ref{prop:rank67points}, and satisfies $I_{40}=0$ by Proposition \ref{prop:I40is0forrank6}. This, together with irreducibility of $I_{40}$, implies that $V(HD) \supseteq V(I_{40})$.
	
	Now if this were a strict inclusion, this would mean that $HD=I_{40} \cdot g$, and so $V(HD) = V(I_{40}) \cup V(g)$, with $V(g) \neq V(I_{40})$. Then $V(g)$ is a codimension one set of cubics lying on the Hessian discriminant. But Propositions \ref{prop:Normalforms}, \ref{prop: SylvHu} and \ref{prop:singHur} show that the set of cubics on $V(HD)$ that cannot be brought in the form (\ref{rank6}) is of codimension greater than one, so we reach a contradiction.
\end{proof}
As pointed out in \cite{Oby}, there is an intuitive reason why $HD$ must be a cube: it follows from the fact that as soon as $f$ lies on the Hessian discriminant, the number of points in $H(f) \cap X_2$ drops from $10$ to $7$, and not to $9$ as expected. This should somehow mean that ``$f$ is a triple zero of $HD$". We can make this intuition precise using the following general lemma about the Hurwitz form:
\begin{mylemma} \label{lem:multihur}
	Let $X \subseteq \PP^n$ be a variety of codimension $d$ and degree $p$, and let $\Hu_X$ be its Hurwitz form in Stiefel coordinates. Denote $\fH=V(\Hu_X) \subseteq \PP^{d(n+1)-1}$, and assume that for a general point $L$ in $\fH$, it holds that $L \cap X$ has exactly 1 double point and $p-2$ simple points.
	Let furthermore $\PP^{\ell}$ be a linear subspace of $\PP^{d(n+1)-1}$ such that for a general point $L$ in $\fH\cap\PP^{\ell}$, it holds that $L \cap X$ has exactly k double points and $p-2k$ simple points.
	Then $\deg(\fH)=k\deg(\fH\cap \PP^{\ell})$, where $\fH\cap \PP^{\ell}$ is the set-theoretic intersection. Hence $\Hu_X|_{\PP^{\ell}}$ is a $k$'th power.
\end{mylemma}
\begin{proof}
	We consider the incidence correspondence 
	\[
	\Phi=\{(p,L)|p \in L \cap X \text{ is a double point}\} \subseteq \PP^n \times \PP^{d(n+1)-1},
	\] 
	and
	\[
	\Phi' = \Phi \cap (\PP^n \times \PP^{\ell}).
	\]
	For a subvariety $Y$ of $\PP^a \times \PP^b$, we write $\deg_2(Y)$ for the number of points in the intersection of $Y$ with a general linear space $\PP^a \times M$ of the correct codimension.
	
	Then \[
	\deg(\fH)=\deg_2(\Phi)=\deg_2(\Phi')=k\deg(\fH \cap \PP^{\ell}).
	\]
\end{proof}
Propositions \ref{prop:Normalforms}, \ref{prop: SylvHu}, \ref{prop:rank67points} and \ref{prop:singHur} imply that if we choose $X=X_2$ and $\PP^{\ell}=\{H(f)|f\in S^3(V^*)\}$, the conditions of Lemma \ref{lem:multihur} are satisfied, hence $HD$ must be a cube.
\bibliography{27bib}{}
\end{document}